\documentclass[11pt]{article}

\usepackage{amsmath}
\usepackage{amssymb}
\usepackage{amsthm}
\usepackage{cite}
\usepackage{enumitem}
\usepackage{geometry}
\usepackage{hyperref}
\usepackage{nicefrac}
\usepackage{tikz}
\usepackage{tikz-cd}
\usepackage{url}

\setenumerate{noitemsep, topsep=0pt}
\newtheorem{theorem}{Theorem}
\newtheorem*{theorem*}{Theorem}
\newtheorem{proposition}[theorem]{Proposition}
\newtheorem*{proposition*}{Proposition}
\newtheorem{lemma}[theorem]{Lemma}
\newtheorem{corollary}[theorem]{Corollary}
\newtheorem{definition}[theorem]{Definition}
\newtheorem{property}[theorem]{Property}
\newtheorem{remark}[theorem]{Remark}
\newtheorem{observation}[theorem]{Observation}
\newtheorem{question}[theorem]{Question}

\newcommand{\tH}{\textnormal{H}}

\title{Simple closed curves, non-kernel homology and Magnus embedding}
\author{Adam Klukowski\footnote{University of Oxford, \texttt{klukowski@maths.ox.ac.uk}}}

\begin{document}

\maketitle

\begin{abstract}
We consider the subspace of the homology of a covering space spanned by lifts of simple closed curves. Our main result is the existence of unbranched covers of surfaces where this is a proper subspace. More generally, for a fixed finite solvable quotient of the fundamental group we exhibit a cover whose homology is not generated by the lifts of curves in the complement of its kernel. We explain how the existing approach of Malestein and Putman (for branched covers) relates to the Magnus embedding, and by doing so we simplify their construction. We then generalise it to unbranched covers by producing embeddings of surface groups into units of certain graded associative algebras, which may be of independent interest.
\end{abstract}

\section{Introduction}\label{sec:introduction}

Consider a finite-degree cover $p \colon \widetilde{\Sigma} \rightarrow \Sigma_{g, n}$ of an orientable surface $\Sigma_{g, n}$ of genus $g$ with $n$ punctures. The \emph{simple closed curve homology} $\tH_1^{\text{scc} | \Sigma_{g, n}} (\widetilde{\Sigma} ; M)$ is defined as the subspace of $\tH_1(\widetilde{\Sigma}; M)$ spanned by the set
\begin{equation*}
\left\{ [\widetilde{\gamma}] \in \tH_1(\widetilde{\Sigma}; M) \ \big| \ \widetilde{\gamma} \text{ is a component of } p^{-1} (\gamma) \text{ for a simple closed curve } \gamma \text{ on } \Sigma_{g, n} \right\}
\end{equation*}
This can be defined for any $\mathbb{Z}$-module $M$.

Our main theorem is:

\begin{theorem}\label{thm:scc_homology}
For any $g \geq 2$, $n \geq 0$, and odd prime $r$ there is a cover $\widetilde{\Sigma} \rightarrow \Sigma_{g, n}$, whose degree is a power of $r$, and such that $\tH_1^{\textnormal{scc} | \Sigma_{g, n}} (\widetilde{\Sigma} ; \mathbb{Q}) \neq \tH_1 (\widetilde{\Sigma} ; \mathbb{Q})$.
\end{theorem}

The question whether simple closed curve homology is the whole of $\tH_1 (\Sigma ; \mathbb{Q})$ was an open problem asked by March\'{e} \cite{homology_generated_by_lifts_of_simple_curves_Marche12} and independently by Looijenga \cite{alggeo_related_to_mcg_Looijenga15}. An affirmative answer for abelian and 2-step nilpotent covers, in the context of graphs instead of surfaces, was given by Farb and Hensel \cite{finite_covers_of_graphs_FarbHensel16}. Koberda and Santharoubane \cite{surf_gps_and_homology_of_covers_via_quantum_reps_KoberdaSantharoubane16} used TQFT methods to find covers of surfaces with $\tH_1^\text{scc} (\widetilde{\Sigma} ; \mathbb{Z}) \neq \tH_1 (\widetilde{\Sigma} ; \mathbb{Z})$, but their method could not rule out the possibility of $\tH_1^\text{scc} (\widetilde{\Sigma} ; \mathbb{Z})$ being finite-index (equivalently, it did not apply to homology with rational coefficients). Another family of covers whose integral simple closed curve homology is a proper submodule was noted by Irmer \cite{lifts_of_sccs_in_regular_coverings_Irmer18}. A particular example of an unbranched cover $\Sigma_{649} \rightarrow \Sigma_2$ with $\tH_1^{\text{scc} | \Sigma_2} (\Sigma_{649} ; \mathbb{Q}) \neq \tH_1 (\Sigma_{649} ; \mathbb{Q})$ was noted by Markovi\'c and To\v{s}i\'c \cite{second_variation_of_hodge_norm_and_higher_Prym_reps_MarkovicTosic22} based on the work of Markovi\'c \cite{unramified_correspondences_and_virtual_properties_of_mcg_Markovic22} and Bogomolov and Tschinkel \cite{unramified_correspondences_BogomolovTschinkel02}.

This article builds on the work of Malestein and Putman \cite{scc_covers_powersubgroups_OutFn_MalesteinPutman19} and earlier ideas of Farb and Hensel \cite{finite_covers_of_graphs_FarbHensel16}. In \cite{scc_covers_powersubgroups_OutFn_MalesteinPutman19} they constructed covers where $\tH_1^\text{scc} (\widetilde{\Sigma} ; \mathbb{Q}) \neq \tH_1 (\widetilde{\Sigma} ; \mathbb{Q})$; however, their construction requires $n \geq 1$, or equivalently, produces branched covers. Here we settle the general case and exhibit unbranched covers with $\tH_1^\text{scc} (\widetilde{\Sigma} ; \mathbb{Q}) \neq \tH_1 (\widetilde{\Sigma} ; \mathbb{Q})$ in all genera. We relate earlier approaches to the Magnus embedding \cite{combinatorial_gptheo_MagnusKarrassSolitar04}. We then produce similar embeddings of surface groups, which may be of independent interest.

\subsection{Homologies of covers and subgroups}

One can ask about the subspaces of homology of a covering space $\widetilde{\Sigma}$ spanned by more general sets. Following \cite{finite_covers_of_graphs_FarbHensel16}, we define an \emph{elevation} $\tilde{\gamma}$ of a loop $\gamma$ on $\Sigma$ to be a minimal concatenation of lifts of $\gamma$ which forms a closed curve in $\widetilde{\Sigma}$. Given a set $\mathcal{O}$ of loops on $\Sigma$ we define $\tH_1^{\mathcal{O} | \Sigma} (\widetilde{\Sigma} ; M)$ to be the subspace of $\tH_1 (\widetilde{\Sigma} ; M)$ spanned by
\begin{equation*}
\left\{ [\tilde{\gamma}] \ \big|\ \tilde{\gamma} \text{ is an elevation of a curve } \gamma \in \mathcal{O} \right\}
\end{equation*}
We present a properness (in the subset/subspace sense) result for a large class of subsets in Theorem \ref{thm:nonkernel_homology}. For a quotient $\theta \colon \pi_1 \Sigma \rightarrow Q$ write $\tH_1^{\theta \neq 1} (\widetilde{\Sigma} ; \mathbb{Q}) = \tH_1^{(\pi_1 \Sigma \setminus \ker \theta) | \Sigma} (\widetilde{\Sigma} ; \mathbb{Q})$ to be the subspace spanned by elevations of curves which are nontrivial in $Q$.

\begin{theorem}\label{thm:nonkernel_homology}
Let $g \geq 2$, $n \geq 0$, and $\theta \colon \pi_1 \Sigma_{g, n} \rightarrow Q$ be a nontrivial finite solvable quotient of odd order. There exists a normal cover $\widetilde{\Sigma} \rightarrow \Sigma_{g, n}$ such that $\tH_1^{\theta \neq 1} (\widetilde{\Sigma} ; \mathbb{Q}) \neq \tH_1 (\widetilde{\Sigma} ; \mathbb{Q})$. Furthermore, we may require that all prime factors of the degree of this cover divide $|Q|$.
\end{theorem}

Let us point out some consequences of Theorem \ref{thm:nonkernel_homology}. They were previously known when the fundamental group is free, and can now be extended to (unpunctured) surfaces. Primitive homology (2) was introduced in \cite{finite_covers_of_graphs_FarbHensel16}. Statement (3) generalises \cite[Theorem C]{scc_covers_powersubgroups_OutFn_MalesteinPutman19}.

\begin{corollary}\label{cor:insufficiencies}
Let $g \geq 2$, $n \geq 0$. For each of the following sets $\mathcal{O}$ of loops on $\Sigma_{g, n}$ there exists a normal cover $\widetilde{\Sigma} \rightarrow \Sigma_{g, n}$ such that $\tH_1^{\mathcal{O} | \Sigma_{g, n}} (\widetilde{\Sigma} ; \mathbb{Q})$ is not the whole of $\tH_1 (\widetilde{\Sigma} ; \mathbb{Q})$.
\begin{enumerate}[label=(\arabic*)]
\item\label{scor:scc_insufficiency} Theorem \ref{thm:scc_homology}: $\mathcal{O}$ consisting of simple closed curves.
\item\label{scor:primitive_insufficiency} Primitive homology: $\mathcal{O}$ consisting of the elements which belong to the generating set in some standard presentation of $\pi_1 \Sigma_{g, n}$.
\item\label{scor:orbit_insufficiency} More generally, any set $\mathcal{O}$ contained in the union of finitely many $\text{Aut} (\pi_1 \Sigma_{g, n})$-orbits.
\end{enumerate}
Furthermore, for any odd prime $r$ we may require the degree of the cover be a power of $r$.
\end{corollary}

Additionally, the following special case of Theorem \ref{thm:nonkernel_homology} generalises Theorem D from \cite{scc_covers_powersubgroups_OutFn_MalesteinPutman19} (where again their original proof required $n \geq 1$). Denote by $\mathcal{O}_d$ the set of all loops which map to nonzero vectors in $\tH_1 (\Sigma_{g, n} ; \tfrac{\mathbb{Z}}{d})$, and define the $d$-primitive homology to be $\tH_1^{d\textnormal{-prim}} (\widetilde{\Sigma} ; \mathbb{Q}) = \tH_1^{\mathcal{O}_d} (\widetilde{\Sigma} ; \mathbb{Q})$.

\begin{proposition}\label{prop:dprimitive_homology}
Let $g \geq 2, n \geq 0$, and $d$ be a square-free integer. There exists a normal cover $\widetilde{\Sigma} \rightarrow \Sigma_{g, n}$ with $\tH_1^{d\textnormal{-prim}} (\widetilde{\Sigma} ; \mathbb{Q}) \neq \tH_1 (\widetilde{\Sigma} ; \mathbb{Q})$. Furthermore, we may require its degree to divide some power of $d$.
\end{proposition}

\begin{remark}
It was asked by Kent \cite{homology_generated_by_lifts_of_simple_curves_Marche12} whether we can take $\mathcal{O}$ to be the set of all curves which do not fill $\Sigma_{g, n}$. This is not answered by Theorem \ref{thm:nonkernel_homology}: if $Q$ is a finite quotient of the fundamental group and $\gamma$ a non-filling curve, then $\gamma^{|Q|}$ is non-filling and trivial in $Q$. After the orignal announcement of this article, Kent's question was answered in the positive by \cite[Theorem D]{generating_homology_of_covers_of_surfaces_BoggiPutmanSalter23}.
\end{remark}

\begin{remark}
When $n \geq 1$ we can remove the restriction of odd order of $Q$ from Theorem \ref{thm:nonkernel_homology}. However, our proof still relies on solvability (this is why it is emphasised, even though every finite group of odd order is necessarily solvable \cite{solvability_of_groups_of_odd_order_FeitThompson63}).
\end{remark}

\paragraph{Algebraic interpretation.} Let $R$ be a finite-index subgroup of $\pi_1 \Sigma_{g, n}$ and $\mathcal{O} \subseteq \pi_1 \Sigma_{g, n}$. Algebraically, the elevation of $\gamma$ is the element $\gamma^m$, where $m$ is the minimal positive integer with $\gamma^m \in R$. Note that topological elevations of (unbased) loop $\gamma$ are precisely the algebraic elevations of (based) loops in the conjugacy class in $\pi_1 \Sigma_{g, n}$ corresponding to $\gamma$. Then $\tH_1^{\mathcal{O}} (R ; M)$ is the subspace of $\tH_1 (R ; M)$ spanned by the set
\begin{equation*}
\left\{ [x^m] \in \tH_1 (R ; M) \ \big| \ x \in \mathcal{O} \text{ and } m \in \mathbb{Z} \text{ such that } x^m \in R \right\}
\end{equation*}
When $R = \pi_1 \widetilde{\Sigma}$ is the subgroup corresponding to the cover and $\mathcal{O}$ is closed under conjugation, then naturally $\tH_1^\mathcal{O} (R ; M) \cong \tH_1^{\mathcal{O}} (\widetilde{\Sigma} ; M)$.

\subsection{Other results}

A couple of interesting statements about free groups was noted by Malestein and Putman \cite[Theorems F and G]{scc_covers_powersubgroups_OutFn_MalesteinPutman19} based on an argument of \cite{surf_gps_and_homology_of_covers_via_quantum_reps_KoberdaSantharoubane16}. Our Theorem \ref{thm:nonkernel_homology} allows to apply their argument to surface groups, giving the following corollary.

\begin{theorem}\label{thm:FG}
Let $g \geq 2$, and $\theta \colon \pi_1 \Sigma_g \rightarrow Q$ be a finite solvable quotient of odd order. There exists an integral linear representation $\rho \colon \pi_1 \Sigma_g \rightarrow \textnormal{GL}_d \mathbb{Z}$ with infinite image such that every element of $\rho (\pi_1 \Sigma_g \setminus \ker \theta)$ has finite order.
\end{theorem}

The statement of Theorem \ref{thm:FG} remains true if we replace the complement of a kernel with any of the special subsets from Corollary \ref{cor:insufficiencies}.

\medskip

Our key technical contribution, which allows to adapt the proofs to surface groups, is the following variation on the Magnus embedding. Here $\pi_1 \Sigma_2 = \langle s_1, t_1, s_2, t_2 | [s_1, t_1] [s_2, t_2] = 1 \rangle$ is the surface group of genus 2, and $\mathbb{Z} [\![ A, B ]\!] [i, j, k]$ is the quaternion algebra of the ring of formal power series in two variables with integer coefficients.

\begin{proposition*}[Lemma \ref{lmm:hom_to_quat}]
There exists a homomorphism $\overline{\tau} \colon \pi_1 \Sigma_2 \rightarrow \mathbb{Z} [\![ A, B ]\!] [i, j, k]^\times$ satisfying
\begin{equation*}
\overline{\tau} (s_1) \equiv 1 + A i, \quad \overline{\tau} (t_1) \equiv 1 + B j, \quad \overline{\tau} (s_2) \equiv 1 + A j, \quad \overline{\tau} (t_2) \equiv 1 + B i \qquad \text{mod } (A, B)^2
\end{equation*}
\end{proposition*}

\subsection{Outline and organisation}

Our approach consists of two steps.
\smallskip

The first step is to construct covers whose homology can see the interesting curves. In Section \ref{ssec:reduction_to_dprimitive} we reduce Theorem \ref{thm:scc_homology} on simple closed curve homology, Corollary \ref{cor:insufficiencies}, and Theorem \ref{thm:nonkernel_homology} to a special case of \ref{thm:nonkernel_homology} -- Proposition \ref{prop:dprimitive_homology} about $d$-primitive homology.

\smallskip

Section \ref{ssec:pf_of_dprimitive} serves as a bridge between the first and second step. Here we view the homology of a cover as a representation of the deck group, and prove Proposition \ref{prop:dprimitive_homology}. To control the representations we use a structure we name \emph{polynomial substitute} (Definition \ref{def:polynomial_substitute}, and existence by Propositions \ref{prop:polynomial_substitute_free} and \ref{prop:polynomial_substitute_surface}). Informally, they are a device for evaluating polynomials on the homology coefficients.

\smallskip

The second step is the construction of polynomial substitutes. We start with the case of punctured surfaces and free groups and prove Proposition \ref{prop:polynomial_substitute_free} in Section \ref{ssec:pf_polynomial_substitute_free}. We translate the proof from \cite{scc_covers_powersubgroups_OutFn_MalesteinPutman19} into the language of Magnus embedding, which leads to significant simplifications. In Section \ref{ssec:pf_polynomial_substitute_surface} we prove Proposition \ref{prop:polynomial_substitute_surface} by carrying out an analogous construction for surface groups.

\bigskip

\noindent\textbf{Acknowledgements.} I wish to thank Vlad Markovi\'{c}, Misha Schmalian, Ognjen To\v{s}i\'{c}, Dawid Kielak, Ric Wade and Andy Putman for helpful suggestions and discussions. I would also like to thank the anonymous reviewer for greatly improving the clarity of this paper. This work was supported by the Simons Foundation.

\section{Relationships between homologies of covers}

The aim of this section is to relate various flavours of cover homology, and eventually reduce the properness of $\tH_1^\text{scc}$, $\tH_1^\mathcal{O}$ and $\tH_1^{\theta \neq 1}$ to the properness of $\tH_1^{d\text{-prim}}$. Our approach is, given a set $\mathcal{O}$ of curves on $\Sigma$, to find a cover $\widehat{\Sigma} \rightarrow \Sigma$ with a nicer set $\widehat{\mathcal{O}}$ (of curves on $\widehat{\Sigma}$) such that all further covers $\widetilde{\Sigma} \rightarrow \widehat{\Sigma}$ satisfy
\begin{equation*}
\tH_1^{\mathcal{O} | \Sigma} \left( \widetilde{\Sigma} ; \mathbb{Q} \right) \ \subseteq \ \tH_1^{\widehat{\mathcal{O}} | \widehat{\Sigma}} \left( \widetilde{\Sigma} ; \mathbb{Q} \right)
\end{equation*}
Then properness of $\widehat{\mathcal{O}}$-homology implies the properness of $\mathcal{O}$-homology.

\subsection{Stability of curve characteristics under taking covers}

Note that each of the properties of interest is stable with respect to taking covers, in the sense that if a curve has a property then so do all its elevations. These facts will be used in Section \ref{ssec:reduction_to_dprimitive} to reduce the general Theorem \ref{thm:nonkernel_homology} to the special case of $d$-primitive homology.

\begin{lemma}\label{lmm:curve_stability}
Let $p \colon \widetilde{\Sigma} \rightarrow \Sigma$ be a finite-degree cover, $\gamma$ a curve on $\Sigma$, and $\tilde{\gamma}$ its elevation to $\widetilde{\Sigma}$. Then:
\begin{enumerate}[label=(\arabic*)]
\item\label{slmm:stability_of_simplicity} If $\gamma$ is simple, then so is $\tilde{\gamma}$.
\item\label{slmm:stability_of_nonfilling} If $\gamma$ does not fill $\Sigma$ then $\tilde{\gamma}$ does not fill $\widetilde{\Sigma}$.
\item\label{slmm:stability_of_rprimitivity} If $p$ is normal, $r$ is a prime, and $[\gamma] \neq 0$ in $\tH_1 ( \Sigma ; \mathbb{F}_r )$, then also $[\tilde{\gamma}] \neq 0$ in $\tH_1 (\widetilde{\Sigma} ; \mathbb{F}_r)$.
\end{enumerate}
\end{lemma}

\begin{proof}
\ref*{slmm:stability_of_simplicity} A self-intersection of $\tilde{\gamma}$ would map to a self-intersection of $\gamma$ under $p$.

\ref*{slmm:stability_of_nonfilling} Let $X$ be a component of $\Sigma \setminus \gamma$ which is not a topological disk. Its preimage $p^{-1} (X)$ is a component of $\widetilde{\Sigma} \setminus p^{-1} (\gamma)$ which is not a disk, so $p^{-1} (\gamma)$ does not fill $\widetilde{\Sigma}$. Since $\tilde{\gamma}$ is a subset of $p^{-1} (\gamma)$, it does not fill $\widetilde{\Sigma}$ either.

\ref*{slmm:stability_of_rprimitivity} Let $\sigma$ be the deck transformation generated by $\gamma$ and define the space $\Sigma' = \nicefrac{\widetilde{\Sigma}}{\langle \sigma \rangle}$. It fits in the sequence
\begin{equation*}\begin{tikzcd}
\widetilde{\Sigma} \arrow[r, "\tilde{p}"] \arrow[rr, bend right, "p"] & \Sigma' \arrow[r, "p'"] & \Sigma
\end{tikzcd}\end{equation*}
of covering maps. The loop $\gamma$ lifts to $\Sigma'$, so $p'_* ([\gamma]) = [\gamma]$, so $[\gamma] \neq 0$ in $\tH_1 (\Sigma' ; \mathbb{F}_r)$.

\smallskip

The cover $\tilde{p} \colon \widetilde{\Sigma} \rightarrow \Sigma'$ is cyclic, so can be realised as a tower of cyclic covers of prime degrees. Therefore, it is now enough to prove Lemma \ref{lmm:curve_stability}\ref*{slmm:stability_of_rprimitivity} for covers where the deck group is the cyclic group of prime order and generated by $\sigma$. If the degree is coprime to $r$, then $p_* ([\gamma]) \neq 0$, so \ref{lmm:curve_stability}\ref*{slmm:stability_of_rprimitivity} holds.

\smallskip

The only remaining case is when $p \colon \widetilde{\Sigma} \rightarrow \Sigma$ is a cyclic cover of degree $r$, and $\sigma$ generates the deck group. Then $\tilde{\gamma}$ is the lift of $\gamma^r$. A part of the five-term exact sequence for $0 \rightarrow \pi_1 \widetilde{\Sigma} \rightarrow \pi_1 \Sigma \rightarrow \tfrac{\mathbb{Z}}{r} \rightarrow 0$ is
\begin{equation*}\begin{tikzcd}
0 = \tH_2 (\tfrac{\mathbb{Z}}{r} ; \mathbb{Z}) \arrow[r] & \tH_1 (\widetilde{\Sigma} ; \mathbb{Z})_{\langle \sigma \rangle} \arrow[r, "p_*"] & \tH_1 (\Sigma ; \mathbb{Z}) \arrow[r] & \tH_1 (\tfrac{\mathbb{Z}}{r} ; \mathbb{Z}) = \tfrac{\mathbb{Z}}{r} \arrow[r] & 0
\end{tikzcd}\end{equation*}
Therefore there exist bases in which $p_*$ becomes the standard inclusion
\begin{equation*}
\tH_1 (\widetilde{\Sigma} ; \mathbb{Z})_{\langle \sigma \rangle} \cong r \mathbb{Z} \oplus \mathbb{Z}^{m-1} \hookrightarrow \mathbb{Z}^m \cong \tH_1 (\Sigma ; \mathbb{Z})
\end{equation*}
Let $\pi \colon \tH_1 (\Sigma ; \mathbb{Z}) \rightarrow \mathbb{Z}$ be the projection onto the first factor. Since $\sigma$ (i.e. ``translation by $\gamma$'') generates the deck group, we have $\pi (\gamma) \in \mathbb{Z} \setminus r \mathbb{Z}$. Therefore $\pi \circ p (\tilde{\gamma}) = \pi (\gamma^r) \in r \mathbb{Z} \setminus r^2 \mathbb{Z}$. Then $r^{-1} \cdot \pi \circ p_*$ is a $\mathbb{Z}$-module homomophism $\tH_1 (\widetilde{\Sigma} ; \mathbb{Z}) \rightarrow \mathbb{Z}$ which sends $[\widetilde{\gamma}]$ to $\mathbb{Z} \setminus r \mathbb{Z}$, so $[\widetilde{\gamma}] \neq 0$ in $\tH_1 (\widetilde{\Sigma} ; \mathbb{F}_r)$.
\end{proof}

\subsection{Reduction to $d$-primitive homology}\label{ssec:reduction_to_dprimitive}

Here we reduce Theorem \ref{thm:nonkernel_homology} and Corollary \ref{cor:insufficiencies} to Proposition \ref{prop:dprimitive_homology} -- properness of $d$-primitive homology. We start by showing how Corollary \ref{cor:insufficiencies} follows from Theorem \ref{thm:nonkernel_homology}.

\begin{proof}[Proof of Corollary \ref{cor:insufficiencies}, assuming Theorem \ref{thm:nonkernel_homology}:] \

Cor \ref{cor:insufficiencies}\ref*{scor:orbit_insufficiency} (orbits) $\Rightarrow$ Cor \ref{cor:insufficiencies}\ref*{scor:scc_insufficiency} (scc): $\text{Mcg} (\Sigma_{g, n})$ acts transitively on non-separating simple closed curves, while separating simple closed curves fall into finitely many orbits (parameterised by genus and partition of punctures \cite[Chapter 1.3]{primer_on_mapping_class_groups_FarbMargalit12}). Since $\text{Mcg} (\Sigma_{g, n})$ acts on $\pi_1 \Sigma_{g, n}$ by outer automorphisms, the set of simple closed curves is contained in a union of finitely many $\text{Aut} (\pi_1 \Sigma_{g, n})$-orbits.

\medskip

Cor \ref{cor:insufficiencies}\ref*{scor:orbit_insufficiency} (orbits) $\Rightarrow$ Cor \ref{cor:insufficiencies}\ref*{scor:primitive_insufficiency} (primitive): $\text{Aut} (\pi_1 \Sigma_{g, n})$ acts transitively on primitive elements.

\medskip

Thm \ref{thm:nonkernel_homology} $\Rightarrow$ Cor \ref{cor:insufficiencies}\ref*{scor:orbit_insufficiency} (orbits): Let $r$ be a prime number. Both free and surface groups are residually $r$-groups, so there exists a finite $r$-group quotient $\theta \colon \pi_1 \Sigma_{g, n} \rightarrow Q$ whose kernel is disjoint from $\mathcal{O}$. Then $\tH_1^{\mathcal{O}} (\widetilde{\Sigma} ; \mathbb{Q}) \subseteq \tH_1^{\theta \neq 1} (\widetilde{\Sigma} ; \mathbb{Q})$ for any cover $\widetilde{\Sigma} \rightarrow \Sigma_{g, n}$. Also, $Q$ is an $r$-group and hence solvable.
\end{proof}

Proposition \ref{prop:dprimitive_homology} about $d$-primitive homology is just a special case of Theorem \ref{thm:nonkernel_homology}, when the quotient $Q$ is $\tH_1 (\Sigma ; \tfrac{\mathbb{Z}}{d})$. Below we show that they are actually equivalent.

\begin{proof}[Proof of Theorem \ref{thm:nonkernel_homology} assuming Proposition \ref{prop:dprimitive_homology}:]
Let $\widehat{\Sigma}$ be the cover corresponding to $\ker \theta$. Denote the derived series of $Q$ as
\begin{equation*}
Q = Q^{(0)} \triangleright Q^{(1)} \triangleright \dots \triangleright Q^{(s)} = \{ 1 \}
\end{equation*}
By refining if necessary, we can assume that $\nicefrac{Q^{(i)}}{Q^{(i+1)}}$ is abelian of square-free order.

Consider $\gamma \in \pi_1 \Sigma_{g, n} \setminus \ker \theta$; then $\theta (\gamma) \in Q^{(i)} \setminus Q^{(i+1)}$ for some $i$. Let $\Sigma'$ be the cover corresponding to $\theta^{-1} (Q^{(i)})$. Then $\widehat{\Sigma}$ factors through $\Sigma'$, the loop $\gamma$ lifts to $\Sigma'$, and is nonzero in the abelian quotient $\nicefrac{Q^{(i)}}{Q^{(i+1)}}$ of $\pi_1 \Sigma'$. Hence $\gamma$ is nonzero in $\tH_1 (\Sigma' ; \mathbb{F}_r)$ for some $r$ dividing $|Q|$. By Lemma \ref{lmm:curve_stability}\ref*{slmm:stability_of_rprimitivity} we know that elevations of $\gamma$ to $\widehat{\Sigma}$ coming from the lift to $\Sigma'$ are nonzero in $\tH_1 (\widehat{\Sigma} ; \mathbb{F}_r)$. But we can obtain all the other elevations by the action of $Q$, so they are all nonzero in $\tH_1 (\widehat{\Sigma} ; \mathbb{F}_r)$.

We have proved that any elevation of any curve in $\pi_1 \Sigma_{g, n} \setminus \ker \theta$ to $\widehat{\Sigma}$ is non-trivial in $\tH_1 (\widehat{\Sigma} ; \tfrac{\mathbb{Z}}{d})$, where $d$ is the radical of $|Q|$. Therefore, for any further cover $\widetilde{\Sigma} \rightarrow \widehat{\Sigma}$ we have
\begin{equation*}
\tH_1^{\theta \neq 1 | \Sigma_{g, n}} (\widetilde{\Sigma} ; \mathbb{Q}) \subseteq \tH_1^{d\text{-prim} | \widehat{\Sigma}} (\widetilde{\Sigma} ; \mathbb{Q})
\end{equation*}
so a witness to properness of $d$-primitive homology relative to $\widehat{\Sigma}$ (Proposition \ref{prop:dprimitive_homology}) is also a witness to Theorem \ref{thm:nonkernel_homology} for $\Sigma_{g, n}$.
\end{proof}

\subsection{Properness of $d$-primitive homology}\label{ssec:pf_of_dprimitive}

Here we prove Proposition \ref{prop:dprimitive_homology}. This will complete the proof of Theorem \ref{thm:nonkernel_homology} given in Section \ref{ssec:reduction_to_dprimitive}. We start by constructing a group with certain special properties.

\begin{proposition}\label{prop:special_quotient}
Suppose $g \geq 2, n \geq 0$, and $d > 1$ is an odd square-free integer. Write $m = 2 g + \max \{ 0, n - 1 \}$. There exists a finite quotient $\rho \colon \pi_1 \Sigma_{g, n} \rightarrow G$, whose size divides some power of $d$, a homomorphism $\alpha \colon G \rightarrow \left( \tfrac{\mathbb{Z}}{d} \right)^m$, and a nontrivial irreducible rational representation $V$ of $G$ with the following properties.
\begin{enumerate}
\item The composition $\alpha \circ \rho$ equals the abelianisation map $\pi_1 \Sigma_{g, n} \rightarrow \tH_1 (\Sigma_{g, n} ; \tfrac{\mathbb{Z}}{d}) = \left( \tfrac{\mathbb{Z}}{d} \right)^m$.
\item Every element of $G \setminus \ker \alpha$ has no nonzero fixed vectors in $V$.
\end{enumerate}
\end{proposition}

Later, when deducing \ref{prop:dprimitive_homology}, we will arrange $G$ to be the deck group, and the representation $V$ will give rise to a subspace complement to $\tH_1^{d\text{-prim}}$. To control the action of elements of $G$ based on their abelianisation we will use a device we name \emph{polynomial substitute} and define below.

\begin{definition}[polynomial substitute]\label{def:polynomial_substitute}
Let $r$ be a prime, $F$ a group, and $n$ be the dimension of $\tH_1 (F ; \mathbb{F}_r)$. A characteristic $r$ degree-$r^k$ \emph{polynomial substitute} for $F$ is a tuple $(G, C, \rho, \alpha)$ consisting of a finite $r$-group $G$, a central subgroup $C \leq G$, and homomorphisms $\rho \colon F \rightarrow G$, $\alpha \colon G \rightarrow \mathbb{F}_r^n$ satisfying the following properties.
\begin{enumerate}[label=(\alph*)]
\item\label{sdef:compatibility_with_abelianisation} The composition $\alpha \circ \rho$ equals the mod-$r$ abelianisation $F \rightarrow \mathbb{F}_r^n = \tH_1 (F ; \mathbb{F}_r)$.
\item\label{sdef:landing_in_centre} For any $x \in G$ we have $x^{r^k} \in C$.
\item\label{sdef:polynomial_substitute_property} For any homogeneous polynomial $P \in \mathbb{F}_r [a_1, \dots, a_n]$ of degree $r^k$ there is a homomorphism $\Psi_P \colon C \rightarrow \mathbb{F}_r$ of additive groups such that $\Psi_P ( x^{r^k}) = P (\alpha(x))$ for all $x \in G$.
\end{enumerate}
\end{definition}

\begin{remark}\label{rmk:all_fns_are_polys}
Any homogeneous function $\mathbb{F}_r^m \rightarrow \mathbb{F}_r$ can be expressed as a polynomial of sufficiently large degree. One way to see this is via Lagrange interpolation. If $v\in\mathbb{F}_r^m$ is nonzero, then the complement of the ray $\mathbb{F}_r^\times\cdot v$ is a union of finitely many codimension-1 subspaces, and the product of their defining equations is a polynomial which vanishes everywhere except nonzero multiples of $v$. By taking linear combinations of such products, we can obtain a homogeneous polynomial with prescribed values at a given set of representatives for rays in $\mathbb{F}_r^m\setminus\{0\}$.
\end{remark}

Before using polynomial substitutes as witnesses to Proposition \ref{prop:special_quotient}, we need to know that they exist. We affirm this in Propositions \ref{prop:polynomial_substitute_free} and \ref{prop:polynomial_substitute_surface} below.

\begin{proposition}\label{prop:polynomial_substitute_free}
Let $r$ be prime and $k, n \in \mathbb{N}_+$. There exists a degree $r^k$ polynomial substitute for the free group $F_n$ of rank $n$.
\end{proposition}

\begin{proposition}\label{prop:polynomial_substitute_surface}
Let $r$ be an odd prime and $k \in \mathbb{N}_+, g \geq 2$. There exists a degree $r^k$ polynomial substitute for the genus $g$ surface group $\pi_1 \Sigma_g$.
\end{proposition}

Proofs of Propositions \ref{prop:polynomial_substitute_free} and \ref{prop:polynomial_substitute_surface} are delayed until Section \ref{sec:polynomial_substitutes}. Details of their construction will not matter for Proposition \ref{prop:special_quotient}, knowing the conditions of Definition \ref{def:polynomial_substitute} will be enough to complete the proof of \ref{prop:special_quotient}.

\begin{proof}[Proof of Proposition \ref{prop:special_quotient}, assuming Propositions \ref{prop:polynomial_substitute_free} and \ref{prop:polynomial_substitute_surface}:] \

\textbf{Construction of $\rho, G, \alpha$.} Factorise $d = \prod_i r_i$. Let $k_i$ be large enough integers, which we will specify in a moment. By Propositions \ref{prop:polynomial_substitute_free} and \ref{prop:polynomial_substitute_surface}, for each factor $r_i$ there exists a degree-$r_i^{k_i}$ polynomial substitute $(G_i, C_i, \rho_i, \alpha_i)$ for $\pi_1 \Sigma_{g, n}$. Define
\begin{align*}
G =& \prod_i G_i & C =& \prod_i C_i \leq Z (G) \\
\rho =& \prod_i \rho_i \colon \pi_1 \Sigma_{g, n} \rightarrow G & \alpha =& \prod_i \alpha_i \colon G \rightarrow \prod_i \mathbb{F}_{r_i}^m \cong \left( \tfrac{\mathbb{Z}}{d} \right)^m
\end{align*}

Now, let $k_i$ be large enough so that there exists a homogeneous polynomial $P_i \in \mathbb{F}_{r_i} [a_1, \dots, a_m]$ of degree $r_i^{k_i}$ which vanishes nowhere on $\mathbb{F}_{r_i}^m \setminus \{ 0 \}$ (see Remark \ref{rmk:all_fns_are_polys}). By the definition of polynomial substitute, there exists a group homomorphism $\Psi_i \colon C_i \rightarrow \mathbb{F}_{r_i}$ satisfying $\Psi_i \big( x^{r_i^{k_i}} \big) = P_i \big( \alpha_i (x) \big)$ for all $x \in G_i$. Define
\begin{equation*}
\Psi \colon C \rightarrow \tfrac{\mathbb{Z}}{d} \qquad \Psi = \sum_i \prod_{j \neq i} r_j \cdot \Psi_i
\end{equation*}
The homomorphism $\Psi$ has the following property
\begin{equation}\label{eq:hom_psi_non_vanishing}
\text{If } x \in G \setminus \ker \alpha \text{, then some power of } x \text{ lies in } C \setminus \ker \Psi.
\end{equation}
Indeed, suppose $x \in G \setminus \ker \alpha$. Then $\alpha_i (x)$ must be a nonzero vector in $\mathbb{F}_{r_i}^m$ for some factor $r_i$. Write $s = \prod_{j \neq i} r_j^{k_j}$ and $t = \prod_{j \neq i} r_j$. Using the definition of $\Psi_i$, and that $x^{r_i^{k_i} s} \in C$, we obtain congruences
\begin{equation*}
\Psi \left( x^{r_i^{k_i} s} \right) \equiv t \Psi_i \left( x^{r_i^{k_i} s} \right) \equiv t P_i \left( \alpha \left( x^s \right) \right) \equiv t P_i \left( s \cdot \alpha (x) \right) \quad \mod r_i
\end{equation*}
where by $s \cdot \alpha(x)$ we mean scalar-by-vector multiplication. We chose $P_i$ to be nonzero on every nonzero vector in $\mathbb{F}_{r_i}^m$, so $\Psi \left( x^{r_i^{k_i} s} \right)$ is nonzero modulo $r_i$, and thus also modulo $d$. This completes the proof of property (\ref{eq:hom_psi_non_vanishing}).

\textbf{Construction of the representation $V$.} Having $\Psi$ with property (\ref{eq:hom_psi_non_vanishing}) we are ready to construct the representation $V$. Let $\omega$ be a primitive $d$-th root of unity and let $W = \mathbb{Q} [\omega]$ be the cyclotomic field. We make $W$ into a representation of $C$ by letting $c \in C$ act by $\omega^{\Psi(c)}$. Then we induce $V' = \text{Ind}^G_C W$.

Suppose that $x \in G \setminus \ker \alpha$ fixes some $v \in V'$. By the property (\ref{eq:hom_psi_non_vanishing}), $x$ has a power $c \in C \setminus \ker \Psi$, which must also fix $v$. Choosing a set of coset representatives $\Lambda$ for $C$ in $G$ allows us to describe the induced representation explicitly as $V' = \bigoplus_{\lambda \in \Lambda} \lambda \cdot W$. Using this description we can write $v = \sum_{\lambda \in \Lambda} \lambda \cdot v_\lambda$. Since $c$ is central, we have
\begin{equation*}
v = c \cdot v = \sum_{\lambda \in \Lambda} \lambda \cdot (c \cdot v_\lambda) = \sum_{\lambda \in \Lambda} \lambda \cdot \omega^{\Psi(c)} v_\lambda = \omega^{\Psi(c)} v
\end{equation*}
But $\Psi(c) \not\equiv 0 \ ( \!\!\!\! \mod d)$, so $v$ must be the zero vector. Thus we can take $V$ to be any irreducible factor of $V'$.
\end{proof}

Now we have to relate the homology of a cover to the representation theory of the deck group. We will use the following two lemmas. They are classical results of Gasch\"utz \cite{modulare_Darstellungen_Gaschutz54} and Chevalley-Weil \cite{verhalten_indegrale_1gattung_automorphismem_funktionenkorpers_ChevalleyWeilHecke34}. Below $G$ acts on $R$ by conjugation and on $k[G]$ by multiplication.

\begin{lemma}[Gash\"utz]\label{lmm:cover_homology_module_free}
Let $R \triangleleft F_n$ be finite-index, $G = \nicefrac{F_n}{R}$, and $k$ be a field of characteristic 0. Then $\tH_1 (R ; k) \cong k \oplus k[G]^{n-1}$ as a $G$-module.
\end{lemma}

\begin{lemma}[Chevalley-Weil]\label{lmm:cover_homology_module_surface}
Let $R \triangleleft \pi_1 \Sigma_g$ be finite-index, $G = \nicefrac{\pi_1 \Sigma_g}{R}$, and $k$ be a field of characteristic 0. Then $\tH_1 (R ; k) \cong k^2 \oplus \left( k[G] \right)^{2g-2}$ as a $G$-module.
\end{lemma}

We are now ready to prove the properness of $d$-primitive homology. The key property of $(G, \rho, \alpha, V)$ is that $V$ is a representation of $\pi_1 \Sigma_{g, n}$ which factors through a finite group and where $d$-primitive elements have no nontrivial fixed vectors.

\begin{proof}[Proof of Proposition \ref{prop:dprimitive_homology} assuming Proposition \ref{prop:special_quotient}:]
Consider the group $\pi_1\Sigma_{g,n}$ and a square-free integer $d$. Let $(G,\rho,\alpha,V)$ be the structure provided by Proposition \ref{prop:special_quotient}. Set $R = \ker \rho \leq \pi_1 \Sigma_{g, n}$. By Lemmas \ref{lmm:cover_homology_module_free} and \ref{lmm:cover_homology_module_surface}, $\tH_1 (R ; \mathbb{Q})$ contains $\mathbb{Q}[G]$ as a direct summand, and hence it also contains $V$.

Take any $d$-primitive $\gamma\in\pi_1\Sigma_{g,n}$, and suppose that $\gamma^m\in R$. We want to prove that projection of the vector $[\gamma^m]\in\tH_1(R;\mathbb{Q})$ onto sub-$G$-representation $V$ is zero. To this end, let us examine the action of $\rho(\gamma)\in G$ on it. On the one hand, $\rho(\gamma)\curvearrowright\tH_1(R;\mathbb{Q})$ corresponds to conjugation by $\gamma$, so it fixes $[\gamma^m]$. On the other hand, $\gamma$ is $d$-primitive, so $\rho(\gamma)\in G\setminus\ker\alpha$, and from the defining property of $V$ it follows that projections of invariants of $\rho(\gamma)$ onto $V$ are zero.

The above argument holds for any generator of $d$-primitive homology, so intersection of $\tH_1^{d\text{-prim}}(R;\mathbb{Q})$ with $V$ is the zero subspace. In particular, $\tH_1^{d\text{-prim}}(R;\mathbb{Q})\neq\tH_1(R;\mathbb{Q})$. The cover $\widetilde{\Sigma}\rightarrow\Sigma$ corresponding to $R\leq\pi_1\Sigma_{g, n}$ is a witness to Proposition \ref{prop:dprimitive_homology}.
\end{proof}

\begin{remark}
We argued that, when no simple closed curve fixes a nonzero vector, then the irreducible representation of the deck group does not appear in $\tH_1^\textnormal{scc}$. It was noted in \cite{finite_covers_of_graphs_FarbHensel16} that the converse also holds.
\end{remark}

\section{Existence of polynomial substitutes}\label{sec:polynomial_substitutes}

This section is devoted to constructing polynomial substitutes (Definition \ref{def:polynomial_substitute}) for free and surface groups. This is the last piece we need to complete the proof of Proposition \ref{prop:special_quotient} and hence of all the theorems stated in Section \ref{sec:introduction}.

In Section \ref{ssec:pf_polynomial_substitute_free} we prove Proposition \ref{prop:polynomial_substitute_free} -- the existence of polynomial substitutes for free groups. This completes the proof of Proposition \ref{prop:special_quotient} in the punctured case $n \geq 1$. The construction relies on the Magnus embedding of a free group into a group of units of a certain associative algebra.

In Section \ref{ssec:pf_polynomial_substitute_surface} we prove Proposition \ref{prop:polynomial_substitute_surface} -- the existence of polynomial substitutes for surface groups. This will complete the proof of \ref{prop:special_quotient} in the unpunctured case $n = 0$. The idea is to mimic the simpler proof of \ref{prop:polynomial_substitute_free}, but work with algebras which allow us to preserve the surface relation.

\subsection{Free groups and proof of Proposition \ref{prop:polynomial_substitute_free}}\label{ssec:pf_polynomial_substitute_free}

In this section we construct polynomial substitutes for free groups and prove Proposition \ref{prop:polynomial_substitute_free}. This will complete the proof of Proposition \ref{prop:special_quotient} in the punctured case $n \geq 1$. The group $G$ in Definition \ref{def:polynomial_substitute} will be a certain quotient of the group of units of a free associative algebra.

\paragraph{Magnus embedding.} Let $F_n = \langle s_1, \dots, s_n \rangle$ be the free group of rank $n$, and $\mathbb{Z} \langle\!\langle X_1, \dots, X_n \rangle\!\rangle$ be the $(X_1, \dots, X_n)$-adic completion of the free associative $\mathbb{Z}$-algebra on $n$ generators. Explicitly, it consists of formal power series in non-commutative variables $X_1, \dots, X_n$ with integer coefficients. We have the following

\begin{lemma}[Magnus embedding]\label{lmm:Magnus_embedding}
There is a homomorphism $\overline{\rho}$ from $F_n$ to the multiplicative group of units $\mathbb{Z} \langle\!\langle X_1, \dots, X_n \rangle\!\rangle^\times$ which sends $s_i$ to $1 + X_i$.
\end{lemma}

The only nontrivial part of the proof is verifying that $1 + X_i$ is actually a unit. Its multiplicative inverse is given by $1 + \sum_{k \geq 1} (-1)^k X_i^k$. More generally, any series with free term $\pm 1$ is a unit, with inverse given by a formal geometric power series. Although we will not need this, we remark that the kernel of $\overline{\rho}$ is in fact trivial \cite{combinatorial_gptheo_MagnusKarrassSolitar04}.

\paragraph{Construction of a polynomial substitute for a free group.} Define
\begin{equation*}
\mathcal{F} = \frac{\mathbb{Z} \langle\!\langle X_1, \dots, X_n \rangle\!\rangle}{(r) + (X_1, \dots, X_n)^{r^k+1}}
\end{equation*}
Explicitly, $\mathcal{F}$ consists of polynomials in non-commuting variables $X_1, \dots, X_n$ with $\mathbb{F}_r$-coefficients, where every product of more than $r^k$ variables vanishes. This is a finite $\mathbb{F}_r$-algebra.

We define $G$ to be the set of polynomials with constant term equal to one
\begin{equation}\label{eq:G_F}
G = 1 + (X_1, \dots, X_n) = \{ P \in \mathcal{F} \ |\ P(0, \dots, 0) = 1 \}
\end{equation}
This is a subgroup of the multiplicative group of units $\mathcal{F}^\times$. Note that $\mathcal{F}^\times \cong \mathbb{F}_r^\times \times G$, and that the order of $G$ is a power of $r$.

We define $\rho$ as the composition of the quotient $\mathbb{Z} \langle\!\langle X_1, \dots, X_n \rangle\!\rangle \rightarrow \mathcal{F}$ with $\overline{\rho}$. Explicitly, it sends $s_i$ to $1 + X_i \in G$.

The group $G$ surjects onto $\mathbb{F}_r^n$ via forgetting degree two and higher and subtracting the constant term
\begin{equation}\label{eq:abelianisation}
\alpha \colon 1 + a_1X_1 + \dots + a_nX_n + P \mapsto \big( a_i \big)_{1 \leq i \leq n} \in \mathbb{F}_r^n
\end{equation}
where $P\in(X_1,\dots,X_n)^2$. The maps $\rho$ and $\alpha$ are compatible, in the sense that $\alpha \circ \rho$ is the mod-$r$ abelianisation $F_n \rightarrow \tH_1 (F_n ; \mathbb{F}_r)$. This verifies condition \ref*{sdef:compatibility_with_abelianisation} of definition \ref{def:polynomial_substitute}.

Inside $\mathcal{F}$ we also have
\begin{equation}\label{eq:C_F}
C = 1 + (X_1, \dots, X_n)^{r^k} = \left\{ 1 + P \ \big|\ P \in \mathbb{F}_r[X_1,\dots,X_n] \text{ homogeneous of degree } r^k \right\}
\end{equation}
All terms of degree at least $r^k+1$ vanish in $\mathcal{F}$, so $C$ is central in $\mathcal{F}$ as a subalgebra, so it is central in $G$ as a subgroup. Let us describe the group structure of $C$. If $P_1,P_2\in\mathbb{F}_r[X_1,\dots,X_n]$ are homogeneous of degree $r^k$, then $(1+P_1)(1+P_2)=1+P_1+P_2+P_1P_2=1+P_1+P_2$ holds in $\mathcal{F}$. Therefore we have a group isomorphism
\begin{align}\label{eq:C_iso}\begin{split}
\varphi\colon C&\cong\left\{P\in\mathbb{F}_r[X_1,\dots,X_n]\ |\ P \text{ homogeneous of degree }r^k\right\}\\
\varphi\colon 1+P&\mapsto P
\end{split}\end{align}
between $C$ (with its multiplicative group structure inherited from $\mathcal{F}$ and $G$) and the linear space of homogeneous polynomials of degree $r^k$. Note that this space has dimension $n^{r^k}$ and a canonical basis consisting of monomials of degree $r^k$.

\paragraph{Verification of the polynomial substitute property.} It remains to check the conditions \ref*{sdef:landing_in_centre} and \ref*{sdef:polynomial_substitute_property} of Definition \ref{def:polynomial_substitute}. Consider $x=1+\sum_{i=1}^na_iX_i+P$ for some $P\in(X_1,\dots,X_n)^2$. Raising it to the $r^k$-th power in $\mathcal{F}$ gives
\begin{align}\label{eq:power_monomials}\begin{split}
x^{r^k}&=\left(1+\sum_{i=1}^na_iX_i+P\right)^{r^k}=\\
&=1+\left(\sum_{i=1}^na_iX_i+P\right)^{r^k}=\\
&=1+\left(\sum_{i=1}^na_iX_i\right)^{r^k}=\\
&=1+\sum_{i_1, \dots, i_{r^k}}\prod_{l=1}^{r^k}a_{i_l}\cdot\prod_{l=1}^{r^k} X_{i_l}
\end{split}\end{align}
The result belongs to $C$, which confirms (b).

Let us view equation (\ref{eq:power_monomials}) through the isomorphism from equation (\ref{eq:C_iso}). By examining the RHS of (\ref{eq:power_monomials}), and using monomials in $X_1,\dots,X_n$ as the basis, we see that $\varphi(x^{r^k})=\alpha(x)^{\otimes r^k}$, with $\alpha$ defined in (\ref{eq:abelianisation}) (resembling the Veronese mapping). In particular, every monomial in $a_1,\dots,a_n$ of degree $r^k$ appears as one of the coordinates of $\varphi(x^{r^k})$. Since every polynomial in $a_1,\dots,a_n$ is a linear combination of monomials, we can obtain any homogeneous degree-$r^k$ polynomial in variables $a_1,\dots,a_n$ by taking linear combinations of coordinate projections of $\varphi(x^{r^k})$. This confirms the condition (c).

\subsection{Surface groups and proof of Proposition \ref{prop:polynomial_substitute_surface}}\label{ssec:pf_polynomial_substitute_surface}

In this section we construct polynomial substitutes for surface groups and prove Proposition \ref{prop:polynomial_substitute_surface}. This will complete the proof of Proposition \ref{prop:special_quotient} in the non-punctured case $n = 0$.

Our proof strategy will be to mimic the construction for free groups from Section \ref{ssec:pf_polynomial_substitute_free}. We will replace the vanilla Magnus embedding \ref{lmm:Magnus_embedding} with homomorphisms coming from two other algebras, each of which preserves the relation in surface groups. In Section \ref{sssec:RAAG_algebra} we showcase the algebra $\overline{\mathcal{M}}$, which will generate \emph{mixed} monomials -- those which contain factors from more than one standard pair of generators (e.g. $a_1 b_1 a_2$ or $b_i b_j$ for $i \neq j$). In Section \ref{sssec:quaternion_algebra} we construct the algebra $\overline{\mathcal{H}}$, which will supply \emph{non-mixed} monomials (e.g. $a_1^u b_1^v$). We combine them and finalise the proof of \ref{prop:polynomial_substitute_surface} in Section \ref{sssec:assembling_ps_surface}.

We use the standard presentation $\pi_1 \Sigma_g = \langle s_1, t_1, \dots, s_g, t_g \ |\ \prod_i [s_i, t_i] = 1 \rangle$, and write $a_i (s_j) = b_i (t_j) = \delta_{i j}, a_i (t_j) = b_i (s_j) = 0$ for the standard basis of $\tH^1 (\Sigma_g ; \mathbb{F}_r)$.

\subsubsection{RAAG algebra and mixed terms}\label{sssec:RAAG_algebra}

In this section we exhibit the algebra $\mathcal{M}$ which admits an interesting mapping $\pi_1 \Sigma_g \rightarrow \mathcal{M}^\times$.

\paragraph{Algebra $\overline{\mathcal{M}}$.} Define
\begin{equation*}
\overline{\mathcal{M}} = \frac{\mathbb{Z} \langle\!\langle X_1, Y_1, \dots, X_g, Y_g \rangle\!\rangle}{(X_iY_i, Y_iX_i)_{1 \leq i \leq g}}
\end{equation*}
that is, the associative algebra of formal power series in non-commutative variables $X_i, Y_i$ with integer coefficients, where we ignore all terms with adjacent $X_i, Y_i$. To be clear, we retain $X_iY_j$ and $Y_j X_i$ when $i \neq j$.

\begin{observation}
There exists a homomorphism $\overline{\rho} \colon \pi_1 \Sigma_g \rightarrow \overline{\mathcal{M}}^\times$ which sends $s_i \mapsto 1 + X_i$ and $t_i \mapsto 1 + Y_i$.
\end{observation}

\begin{proof}
The symbols $X_i, Y_i$ commute in the algebra $\overline{\mathcal{M}}$, so $1 + X_i, 1 + Y_i$ commute in the multiplicative group $\overline{\mathcal{M}}^\times$. In particular the surface relation is mapped to 1.
\end{proof}

\medskip

Analogously as before, define
\begin{equation*}
\mathcal{M} = \frac{\overline{\mathcal{M}}}{(r) + (X_i, Y_i)^{r^k+1}}
\end{equation*}
This is a finite $\mathbb{F}_r$-algebra. We define $G_\mathcal{M}, \rho_\mathcal{M}, \alpha_\mathcal{M}, C_\mathcal{M}$ analogously as before: $G_\mathcal{M}$ is the subgroup of series with trailing term 1 (see \ref{eq:G_F}), $\rho$ is the composition of the quotient $\overline{\mathcal{M}} \rightarrow \mathcal{M}$ with $\overline{\rho}$, $\alpha_\mathcal{M}$ comes from forgetting quadratic and higher terms (see \ref{eq:abelianisation}) and $C_\mathcal{M}$ contains series which are 1 plus terms of degree $r^k$ (see \ref{eq:C_F}). Note this is compatible with abelianisation (condition \ref*{sdef:compatibility_with_abelianisation} of definition \ref{def:polynomial_substitute}), in the sense that $\alpha_\mathcal{M} \circ \rho_\mathcal{M}$ is the canonical map $\pi_1 \Sigma_g \rightarrow \tH_1 (\Sigma_g ; \mathbb{F}_r)$. Every $r^k$-th power lands in $C_\mathcal{M}$, so $(G, C, \rho, \alpha)$ satisfy condition \ref{def:polynomial_substitute}\ref*{sdef:landing_in_centre}.

\paragraph{Analogue of the polynomial substitute property.} The group $G_\mathcal{M}$ satisfies a weaker analogue of the property \ref{def:polynomial_substitute}\ref*{sdef:polynomial_substitute_property}. We say a monomial $\prod_j a_j^{u_j} b_j^{v_j}$ is \emph{mixed} if $u_i + v_i \geq 1$ for at least two different $i$; in other words, it contains, with nonzero exponents, two factors representing different standard generator pairs.

\begin{property}\label{pty:polyhom_mixed_surface}
Let $P$ be a degree $r^k$ mixed monomial in $a_1, b_1, \dots, a_g, b_g$. There exists a group homomorphism $\Psi_P \colon C_\mathcal{M} \rightarrow \mathbb{F}_r$, such that $\Psi_P \big( x^{r^k} \big) = P \big( \alpha_\mathcal{M} (x) \big)$ for all $x \in G_\mathcal{M}$.
\end{property}

We will remediate this weakening later by using complementary Property \ref{pty:polyhom_nonmixed_surface}.

\begin{proof}[Proof of Property \ref{pty:polyhom_mixed_surface}:]
Write
\begin{equation*}
P = \prod_{i=1}^g a_i^{u_i} b_i^{v_i}
\end{equation*}
Since $P$ is mixed, either $P = \prod_i a_i^{u_i}$ or $P = \prod_i b_i^{v_i}$ or there exist $i_0 \neq j_0$ with $u_{i_0}, v_{j_0} > 0$. In the first two cases set $Q = \prod_i X_i^{u_i}$ or $Q = \prod_i Y_i^{v_i}$ respectively. In the third case set
\begin{equation*}
Q = \prod_{i \neq i_0} X_i^{u_i} \cdot X_{i_0}^{u_{i_0}} \cdot Y_{j_0}^{v_{j_0}} \cdot \prod_{j \neq j_0} Y_j^{v_j}
\end{equation*}
In either case, such $Q$ has no adjacent $X_i Y_i$ or $Y_i X_i$.

Recall that the group $C_\mathcal{M}$ is isomorphic to the $\mathbb{F}_r$-vector space with the basis $\{1+R\}_R$, where $R$ ranges over monomials of total degree $r^k$ which do not have $X_i,Y_i$ as adjacent factors. We argued that $Q$ does not have neighbouring $X_i,Y_i$, which means that $1+Q$ is an element of this basis. Let $\Psi_P$ be the coordinate projection $C_\mathcal{M} \rightarrow \mathbb{F}_r$ onto the 1-dimensional subspace $1+Q\mathbb{F}_r$ spanned by $1+Q$. The exponents of $Q$ match the corresponding exponents of $P$, so $\Psi_P$ satisfies $\Psi_P (x^{r^k}) = P \big( \alpha(x) \big)$ for all $x \in G_\mathcal{M}$ (compare equation (\ref{eq:power_monomials})).
\end{proof}

\subsubsection{Quaternion algebra and products of intersecting generators}\label{sssec:quaternion_algebra}

Here we construct the algebra $\mathcal{H}$ and an interesting mapping $\pi_1 \Sigma_2 \rightarrow \mathcal{H}^\times$.

\paragraph{Algebra $\overline{\mathcal{H}}$ and a homomorphism from the genus two surface group.} Consider the commutative ring $\mathbb{Z} [\![ A, B ]\!]$ and take its quaternion algebra
\begin{equation*}
\overline{\mathcal{H}} = \mathbb{Z} [\![ A, B ]\!] [i, j, k]
\end{equation*}
The variables $A, B$ commute with each other and with $i, j, k$, while $i, j, k$ square to $-1$ and pairwise anti-commute. Note that $\overline{\mathcal{H}}$ can also be considered as the ring of formal power series over the skew-comutative ring $\mathbb{Z} [i, j, k]$.

\medskip

First we establish the existence of an interesting homomorphism from a surface group of genus two into its group of units.

\begin{lemma}\label{lmm:hom_to_quat}
There exists a group homomorphism $\overline{\tau} \colon \pi_1 \Sigma_2 \rightarrow \overline{\mathcal{H}}^\times$ satisfying
\begin{equation}\label{eq:quathom_linearisation}
\overline{\tau} (s_1) \equiv 1 + A i, \quad \overline{\tau} (t_1) \equiv 1 + B j, \quad \overline{\tau} (s_2) \equiv 1 + A j, \quad \overline{\tau} (t_2) \equiv 1 + B i \qquad \text{mod } (A, B)^2
\end{equation}
\end{lemma}

\begin{proof}
We make an ansatz
\begin{align}\label{eq:quat_ansatz}\begin{split}
\overline{\tau}(s_1) = 1 + A i + E k,& \quad \overline{\tau}(t_1) = 1 + B j, \quad \overline{\tau}(s_2) = 1 + A j - E k, \quad \overline{\tau}(t_2) = 1 + B i \\
&\text{where} \qquad E = \tfrac{1}{2} \left( \big( 1 + 4 A (B - A) \big)^\frac{1}{2} - 1\right)
\end{split}\end{align}
It can be seen from the explicit Taylor expansion of $\sqrt{1 + x}$ that $E$ is a series with integer coefficients and contains only terms of degree two or higher. This means that the ansatz (\ref{eq:quat_ansatz}) is of the desired form (\ref{eq:quathom_linearisation}). To show that it extends to a group homomorphism, it suffices to verify that the images of generators in $\overline{\mathcal{H}}^\times$ satisfy the defining relation of $\pi_1\Sigma_2$. This boils down to an involved but straightforward calculation. We will use the fact that $E$ is a solution to the equation
\begin{equation}\label{eq:E_equation}
A B - E = A^2 + E^2
\end{equation}

Denote group-theoretic and algebra-theoretic commutators as $[\bullet, \bullet]_G$ and $[\bullet, \bullet]_A$ respectively. Under \ref{eq:quat_ansatz} we have
\begin{align}
[ \overline{\tau}(s_1), \overline{\tau}(t_1) ]_G & = (1 + A i + E k) (1 + B j) (1 + A i + E k)^{-1} (1 + B j)^{-1} = \nonumber \\
& = 1 + [A i + E k, B j]_A (1 + A^2 + E^2)^{-1} (1 + B^2)^{-1} (1 - A i - E k) (1 - B j) = \nonumber \\
& = 1 + 2 (1 + A^2 + E^2)^{-1} (1 + B^2)^{-1} B \times \nonumber \\
& \qquad\qquad \times \Big( - (A^2 + E^2) B + (A B - E) i - (A^2 + E^2) j + (A + B E) k \Big) = \nonumber \\
& = 1 + 2 (1 + w)^{-1} (1 + B^2)^{-1} B \big( - w B + w i - w j + (A + B E) k \big) \label{eq:tau_commutator}
\end{align}

where $w = AB - E$ (compare (\ref{eq:E_equation})). Let $\theta$ be the automorphism of $\overline{\mathcal{H}}$ induced by the automorphism of $Q_8$ sending $i \mapsto j, j \mapsto i, k \mapsto -k$. We chose the images of generators which are interchanged by $\theta$, so
\begin{align}\label{eq:tau_commutator2}\begin{split}
[ \overline{\tau}(s_2), \overline{\tau}(t_2) ]_G & = [ \theta( \overline{\tau}(s_1) ), \theta( \overline{\tau}(t_1) ) ]_G = \theta \left( [ \overline{\tau} (s_1), \overline{\tau} (t_1) ]_G \right) = \\
& = 1 + 2 (1 + w)^{-1} (1 + B^2)^{-1} B \big( - w B + w j - w i - (A + B E) k \big)
\end{split}\end{align}

Expressions \ref{eq:tau_commutator} and \ref{eq:tau_commutator2} are quaternionic conjugates, so their product is
\begin{align*}
[ \overline{\tau} (s_1), & \overline{\tau} (t_1) ]_G [ \overline{\tau} (s_2), \overline{\tau} (t_2) ]_G = \\
& = \left( 1 - 2 (1 + w)^{-1} (1 + B^2)^{-1} B^2 w \right)^2 + 4 (1 + w)^{-2} (1 + B^2)^{-2} B^2 (2 w^2 + (A + B E)^2) = \\
& = 1 + 4 (1 + w)^{-2} (1 + B^2)^{-2} B^2 \Big( - (1 + w) (1 + B^2) w + B^2 w^2 + 2 w^2 + (A + B E)^2 \Big)
\end{align*}
This equals 1, provided that the expression
\begin{equation*}
- (1 + w) (1 + B^2) w + B^2 w^2 + 2 w^2 + (A + B E)^2
\end{equation*}
vanishes. Expanding and substituting back $w = A B - E$, this becomes
\begin{equation*}
(1 + B^2) ( -A B + E + A^2 + E^2)
\end{equation*}
which is zero by the equation (\ref{eq:E_equation}). Therefore $[ \overline{\tau} (s_1), \overline{\tau} (t_1) ]_G [ \overline{\tau} (s_2), \overline{\tau} (t_2) ]_G = 1$, so $\overline{\tau}$ is a well-defined group homomorphism.
\end{proof}

\paragraph{Parts of the polynomial substitute.} Let us describe how to build parts of the polynomial substitute from $\overline{\mathcal{H}}$. As before, for an odd prime $r$, we reduce modulo $r$ and degree $r^k+1$ by defining the finite $\mathbb{F}_r$-algebra
\begin{equation*}
\mathcal{H} = \frac{\overline{\mathcal{H}}}{(r) + (A, B)^{r^k+1}}
\end{equation*}
Units $\mathcal{H}^\times$ contain $G_\mathcal{H} = 1 + \mathbb{F}_r \{ Ai, Aj, Bi, Bj \} + (A, B)^2$ and its central subgroup $C_\mathcal{H} = 1 + (A, B)^{r^k}$. Explicitly
\begin{align*}
G_\mathcal{H} =& \Bigg\{ 1 + a_1 A i + b_1 B j + a_2 A j + b_2 B i + \!\! \sum_{\tiny \begin{matrix} u + v \geq 2 \\ l \in \{ 1, i, j, k \} \end{matrix}} c^{(l)}_{u, v} A^u B^v l \Bigg| a_1, b_1, a_2, b_2, c^{(l)}_{u, v} \in \mathbb{F}_r \Bigg\} \\
C_\mathcal{H} =& \Bigg\{ 1 + \! \sum_{\tiny \begin{matrix} u + v = r^k \\ l \in \{ 1, i, j, k \} \end{matrix}} c_{u, v}^{(l)} A^u B^v l \ \Bigg| \ c_{u, v}^{(l)} \in \mathbb{F}_r \Bigg\} \cong \mathbb{F}_r^{4(r^k+1)}
\end{align*}
We define
\begin{equation}\label{eq:tau}
\tau \colon \pi_1 \Sigma_2 \rightarrow G_\mathcal{H} \qquad \tau = q_\mathcal{H} \circ \overline{\tau}
\end{equation}
where $q \colon \overline{\mathcal{H}} \rightarrow \mathcal{H}$ is the quotient map. Forgetting quadratic and higher terms (and constant 1) is a homomorphism $G_\mathcal{H} \rightarrow \mathbb{F}_r^4$ given by
\begin{equation}\label{eq:alpha_H}
\alpha_\mathcal{H} \colon 1 + a_1 A i + b_1 B j + a_2 A j + b_2 B i + P \ \mapsto \ \big( a_1, b_1, a_2, b_2 \big)^\top \in \mathbb{F}_r^4
\end{equation}
for arbitrary $P \in (A,B)^2$. Note that again $\alpha_\mathcal{H} \circ \tau$ is the standard map $\pi_1 \Sigma_2 \rightarrow \tH_1 (\Sigma_2 ; \mathbb{F}_r)$, and $a_l, b_l$ agree with the basis of $\tH^1 (\Sigma_2; \mathbb{F}_r)$ defined in Section \ref{ssec:pf_polynomial_substitute_surface}. Similarly as before, the tuple $(G_\mathcal{H}, C_\mathcal{H}, \tau, \alpha_\mathcal{H})$ satisfies conditions \ref*{sdef:compatibility_with_abelianisation} and \ref*{sdef:landing_in_centre} of Definition \ref{def:polynomial_substitute}. Shortly we will also show a weaker analogue of condition \ref*{sdef:polynomial_substitute_property} from Definition \ref{def:polynomial_substitute}.

\paragraph{Analogue of the polynomial substitute property.} Now we show that $r^k$-th power produces all non-mixed terms from the linearisation. Recall that non-mixed monomials are precisely those of the form $a_i^u b_i^v$, where $u + v = r^k$ and $1 \leq i \leq g$.

\begin{property}\label{pty:polyhom_nonmixed_surface}
Take nonnegative $u, v$ summing to $r^k$. There exists a group homomorphism $\Psi_{a_1^u b_1^v} \colon C_\mathcal{H} \rightarrow \mathbb{F}_r$ and a polynomial $R_{a_1^u b_1^v} \in \mathbb{F}_r [ a_1, b_1, a_2, b_2]$ consisting only of mixed terms, such that for any $x \in G_\mathcal{H}$ we have $\Psi_{a_1^u b_1^v} \big( x^{r^k} \big) = a_1(x)^u b_1(x)^v + R_{a_1^u b_1^v} \big( \alpha_\mathcal{H} (x) \big)$, where $a_1, b_1$ are the coefficients in the mod-$r$ abelianisation defined in equation (\ref{eq:alpha_H}).
\end{property}

This is complementary to Property \ref{pty:polyhom_mixed_surface} of the RAAG algebra $\overline{\mathcal{M}}$. Later we will combine them to satisfy the polynomial substitute condition \ref{def:polynomial_substitute}\ref*{sdef:polynomial_substitute_property}. Proof of Property \ref{pty:polyhom_nonmixed_surface} is the first place where we require the assumption that $r$ is odd.

\begin{proof}[Proof of Property \ref{pty:polyhom_nonmixed_surface}:]
WLOG assume that $u$ is at most $r-1$. We can do this, because whenever $u\geq r$, the values of $a^u b^v$ and $a^{u-(r-1)} b^{v+r-1}$ agree for any $a, b \in \mathbb{F}_r$.

Take $x \in G_\mathcal{H}$ of the form $g = 1 + a_1Ai + b_1Bj + a_2Aj + b_2Bi + (\deg \geq 2)$. Then the image of $x$ under $\alpha_\mathcal{H}$ is the vector ${(a_1, b_1, a_2, b_2)^\top \in \mathbb{F}_r^4}$. Moreover,
\begin{align}
x^{r^k} &= 1 + \big( (a_1A + b_2B)i + (a_2A + b_1B)j \big)^{r^k} = \nonumber \\
&= 1 + \big[ -(a_1A + b_2B)^2 - (a_2A + b_1B)^2 \big]^\frac{r^k-1}{2} \cdot \big[ (a_1A + b_2B)i + (a_2A + b_1B)j \big] \label{eq:power_in_H}
\end{align}
Let us denote by $\approx$ the equality up to mixed terms. Note that arbitrary multiple of $a_1 a_2$, $a_1 b_2$, $b_1 a_2$ or $b_1 b_2$ is necessarily mixed. The terms of degree $r^k$ in the expression (\ref{eq:power_in_H}) are $(-1)^\frac{r^k-1}{2}$ times
\begin{align}
\big[ (a_1A & + b_2B)^2 + (a_2A + b_1B)^2 \big]^\frac{r^k-1}{2} \cdot \big[ (a_1A + b_2B)i + (a_2A + b_1B)j \big] \approx \nonumber \\
\approx & \big[ a_1^2 A^2 + b_1^2 B^2 \big]^\frac{r^k-1}{2} \cdot ( a_1 A i + b_1 B j ) + \big[ a_2^2 A^2 + b_2^2 B^2 \big]^\frac{r^k-1}{2} \cdot ( b_2 B i + a_2 A j ) = \nonumber \\
= & \sum_{w=0}^{\frac{r^k-1}{2}} \begin{pmatrix} \frac{r^k-1}{2} \\ w \end{pmatrix} a_1^{2w} b_1^{r^k-1-2w} A^{2w} B^{r^k-1-2w} \cdot (a_1 A i + b_1 B j) + \nonumber \\
& + \sum_{w=0}^{\frac{r^k-1}{2}} \begin{pmatrix} \frac{r^k-1}{2} \\ w \end{pmatrix} a_2^{2w} b_2^{r^k-1-2w} A^{2w} B^{r^k-1-2w} \cdot (b_2 B i + a_2 A j) \label{eq:nonmixed_coeffs}
\end{align}

Suppose $u$ is odd and $v$ is even; the other case is analogous. Let $u = 2 u_0 + 1, v = 2 v_0$. The coefficient of $A^u B^v i$ in \ref{eq:nonmixed_coeffs} equals
\begin{equation}\label{eq:chosen_mixed_coeff}
\begin{pmatrix} \frac{r^k-1}{2} \\ u_0 \end{pmatrix} a_1^u b_1^v
\end{equation}
We need the binomial coefficient in this expression to be nonzero modulo $r$. Recall Kummer's theorem
\begin{theorem*}[Kummer]
Let $r$ be a prime number. The exponent of the largest power of $r$ dividing the binomial coefficient $\begin{pmatrix}n\\m\end{pmatrix}$ is equal
\begin{equation*}
\frac{S_r(m)+S_r(n-m)-S_r(n)}{r-1}
\end{equation*}
where $S_r$ is the sum of digits in base-$r$ expansion.
\end{theorem*}
We assumed that $u\leq r-1$, so $u_0\leq\tfrac{r-1}{2}$. This means that in base $r$ we have expansions
\begin{align*}
u_0&=u_0\cdot r^0\\
\frac{r^k-1}{2}&=\frac{r-1}{2}\cdot r^{k-1}+\dots+\frac{r-1}{2}\cdot r^1+\frac{r-1}{2}\cdot r^0\\
v_0&=\frac{r-1}{2}\cdot r^{k-1}+\dots+\frac{r-1}{2}\cdot r^1+\left(\frac{r-1}{2}-u_0\right)\cdot r^0
\end{align*}
In particular, $S_r(u_0)+S_r(v_0)=S_r\left(\tfrac{r^k-1}{2}\right)$, so the binomial coefficient in (\ref{eq:chosen_mixed_coeff}) is not divisible by $r$.

Let $\Psi'$ be the coordinate projection $C_\mathcal{H} \rightarrow \mathbb{F}_r$ onto $A^u B^v i$. We have established that
\begin{equation*}
\Psi' \big( x^{r^k} \big) = c \cdot a_1(x)^u b_1(x)^v + R
\end{equation*}
where $c \in \mathbb{F}_r^\times$, and $R$ is a polynomial in $a_1(x),b_1(x),a_2(x),b_2(x)$ involving only mixed terms. Then the homomorphism $\Psi_{a_1^u b_1^v} \stackrel{\text{def}}{=} c^{-1} \Psi'$ witnesses Property \ref{pty:polyhom_nonmixed_surface}.
\end{proof}

\subsubsection{Assembling the polynomial substitute for a surface group}\label{sssec:assembling_ps_surface}

Now we complete the proof of Proposition \ref{prop:polynomial_substitute_surface}. Recall

\begin{proposition*}[restatement of \ref{prop:polynomial_substitute_surface}]
If $r$ is an odd prime and $k \in \mathbb{N}_+, g \geq 2$, then there exists a degree $r^k$ polynomial substitute for the genus $g$ surface group $\pi_1 \Sigma_g$.
\end{proposition*}

Recall that we need to find an $r$-group $G$, central $C \leq G$, and homomorphisms $\rho \colon \pi_1 \Sigma_g \rightarrow G$, $\alpha \colon G \rightarrow \mathbb{F}_r^{2g}$ such that the following hold.
\begin{enumerate}
\item The composition $\alpha \circ \rho \colon \pi_1 \Sigma_g \rightarrow \mathbb{F}_r^{2g} = \tH_1 (\Sigma_g ; \mathbb{F}_r)$ is the Hurewicz map.
\item For any $x \in G$ we have $x^{r^k} \in C$.
\item For any homogeneous $P \in \mathbb{F}_r [a_1, b_1, \dots, a_g, b_g]$ of degree $r^k$ there is a homomorphism $\Psi_P \colon C \rightarrow \mathbb{F}_r$ with $\Psi_P \big( x^{r^k} \big) = P \big( \alpha (x) \big)$.
\end{enumerate}

\medskip

As parts of the construction we will use $(G_\mathcal{M}, C_\mathcal{M}, \rho_\mathcal{M}, \alpha_\mathcal{M})$ from Section \ref{sssec:RAAG_algebra} and $(G_\mathcal{H}, C_\mathcal{H}, \tau, \alpha_\mathcal{H})$ from Section \ref{sssec:quaternion_algebra}. We will assemble them according to the following diagram:
\begin{equation*}\begin{tikzcd}[row sep=8pt]
\pi_1 \Sigma_g \arrow[rr, "\rho_\mathcal{M}"] \arrow[rd, "\prod_i h_i"'] && G_\mathcal{M} \arrow[r, bend left, "x \mapsto x^{r^k}"] \arrow[d, phantom, "\times" description] & C_\mathcal{M} \arrow[l, hook'] \arrow[r, "\Psi_0"] \arrow[d, phantom, "\times" description] & \mathbb{F}_r\\
& \left( \pi_1 \Sigma_2 \right)^g \arrow[r, "\tau^{\times g}"'] & G_\mathcal{H}^g \arrow[r, bend right, "x \mapsto x^{r^k}"'] & C_\mathcal{H}^g \arrow[l, hook'] \arrow[ru, "\sum_i \Psi_i"'] & \\
& \rho \arrow[u, phantom, "\underbrace{\qquad\qquad\qquad\qquad\qquad}"] & G' \arrow[u, phantom, sloped, "="] & C' \arrow[u, phantom, sloped, "="] & \ \arrow[lu, phantom, "\underbrace{\qquad\quad}_{\Psi_P}"]
\end{tikzcd}\end{equation*}
When we construct $\Psi_P$, the top row will supply mixed terms, while the bottom row will provide the remaining non-mixed terms one-by-one.

\begin{proof}[Proof of Proposition \ref{prop:polynomial_substitute_surface}] \

\textbf{Step 1: construction.} Define
\begin{equation*}
G' = G_\mathcal{M} \times G_\mathcal{H}^g
\end{equation*}
and let $f_\mathcal{M}, \{ f_{\mathcal{H}, i} \}_{1 \leq i \leq g}$ be the projections of $G'$ onto each factor.

We need to build a homomorphism $\rho \colon \pi_1 \Sigma_g \rightarrow G'$. For $1 \leq i \leq g$ let $h_i$ be the family of homomorphisms $\pi_1 \Sigma_g \rightarrow \pi_1 \Sigma_2$ forgetting all but two generator pairs, given by
\begin{equation*}
h_i : \left\{ \begin{matrix}
s_{i+a} & \mapsto & s_a & \text{if } a \in \{ 0, 1 \} \\
t_{i+a} & \mapsto & t_a & \text{if } a \in \{ 0, 1 \} \\
s_j, t_j & \mapsto & 1 & \text{if } j \notin \{ i, i+1 \}
\end{matrix} \right.
\end{equation*}
where indices are taken modulo $g$. Geometrically these maps come from collapsing $g-2$ handles of $\Sigma_g$ to points. Define $\rho = \rho_\mathcal{M} \times \prod_{i=1}^g \tau \circ h_i$, where $\tau$ is the homomorphism $\pi_1 \Sigma_2 \rightarrow G_\mathcal{H}$ from equation \ref{eq:tau}. Define $G = \text{im}\ \rho \leq G'$.

We have a map $\alpha = \alpha_\mathcal{M} \circ f_\mathcal{M} \colon G \rightarrow \mathbb{F}_r^{2g}$. Then $\alpha \circ \rho$ equals $\alpha_\mathcal{M} \circ \rho_\mathcal{M}$, which is the standard map $\pi_1 \Sigma_g \rightarrow H_1(\Sigma_g; \mathbb{F}_r)$. Note that on $G$ this is compatible with the abelianisations coming from each factor $G_\mathcal{H}$, in the sense that $\alpha_\mathcal{H} \circ f_{\mathcal{H}, i} \big|_G$ is equal to $\alpha$ followed by the appropriate coordinate projection $\mathbb{F}_r^{2g} \rightarrow \mathbb{F}_r^4$.

The subgroup $C' = C_\mathcal{M} \times C_\mathcal{H}^{2 g}$ is central in $G'$, and contains all the $r^k$-th powers. Define $C = G \cap C'$. We see that $(G, C, \rho, \alpha)$ satisfy the conditions \ref*{sdef:compatibility_with_abelianisation} and \ref*{sdef:landing_in_centre} of definition \ref{def:polynomial_substitute}.

\textbf{Step 2: polynomial substitute property.} Let $P$ be a polynomial of degree $r^k$. Write it in the form
\begin{equation*}
P (a_1, b_1, \dots, a_g, b_g) = Q (a_1, b_1, \dots, a_g, b_g) + \sum_{i=1}^g \sum_{u + v = r^k} c_{i, u, v} a_i^u b_i^v
\end{equation*}
where $c_{i, u, v} \in \mathbb{F}_r$, and $Q$ contains only mixed terms.

By the property \ref{pty:polyhom_nonmixed_surface}, there exist group homomorphisms $\Psi_i \colon C_\mathcal{H} \rightarrow \mathbb{F}_r$ such that for any $x \in G$ with $\alpha (x) = (a_j, b_j)^\top_j$ we have
\begin{equation*}
\Psi_i \circ f_{\mathcal{H}, i} \left( x^{r^k} \right) = R_i \big( \alpha(x) \big) + \sum_{u + v = r^k} c_{i, u, v} a_i(x)^u b_i(x)^v
\end{equation*}
where the polynomials $R_i$ contain only mixed terms. By the property \ref{pty:polyhom_mixed_surface}, there is a group homomorphism $\Psi_0 \colon C_\mathcal{M} \rightarrow \mathbb{F}_r$ satisfying
\begin{equation*}
\Psi_0 \circ f_\mathcal{M} \left( x^{r^k} \right) = \left( Q - \sum_{i=1}^g R_i \right) \big( \alpha(x) \big)
\end{equation*}
Finally, we combine them as
\begin{equation*}
\Psi = \left(\Psi_0 \circ f_\mathcal{M} + \sum_{i=1}^g \Psi_i \circ f_{\mathcal{H}, i} \right) \Bigg|_C
\end{equation*}
By construction $\Psi \big( x^{r^k} \big) = P \big( \alpha(x) \big)$, which confirms the condition \ref{def:polynomial_substitute}(c).
\end{proof}

\section{Comments}

It is natural to ask whether the assumption of solvability in Theorem \ref{thm:nonkernel_homology} is necessary. We ask the following

\begin{question}\label{cnj:outkernel_homology}
Let $g \geq 2$, $n \geq 0$ and $\theta \colon \pi_1 \Sigma_{g, n} \rightarrow Q$ be an arbitrary finite quotient. Does there exist a cover $\widetilde{\Sigma} \rightarrow \Sigma_{g, n}$ with $\tH_1^{\theta \neq 1} (\widetilde{\Sigma} ; \mathbb{Q}) \neq \tH_1 (\widetilde{\Sigma} ; \mathbb{Q})$?
\end{question}

Our proof fundamentally relies on detecting loops using abelian quotients, so solvable $Q$ seems to be the most general we can hope for with the current approach. Let us note a simple obstacle to potential approches to question \ref{cnj:outkernel_homology}:

\begin{observation}
Let $\widetilde{\Sigma} \rightarrow \Sigma$ be a finite-degree normal cover with deck group $G$, and $\theta \colon \pi_1 \Sigma \rightarrow Q$ be a nontrivial finite quotient. If the sizes of $G$ and $Q$ are coprime, then $\tH_1^{\theta \neq 1} (\widetilde{\Sigma} ; \mathbb{Z} ) = \tH_1 (\widetilde{\Sigma} ; \mathbb{Z})$.
\end{observation}

\begin{proof}
Consider $\gamma \in \pi_1 \widetilde{\Sigma}$. We want to show that $[\gamma] \in \tH_1^{\theta \neq 1} (\widetilde{\Sigma}; \mathbb{Z})$. If $\gamma \notin \ker \theta$ then this holds by definition. Otherwise, take any $\lambda \in \pi_1 \Sigma \setminus \ker \theta$. Then $[\gamma]$ can be expressed as the difference of $[\gamma \lambda^{|G|}]$ and $[\lambda^{|G|}]$; both of these classes belong to $\tH_1^{\theta \neq 1} (\widetilde{\Sigma} ; \mathbb{Z})$, so again $[\gamma] \in \tH_1^{\theta \neq 1} (\widetilde{\Sigma}; \mathbb{Z})$.
\end{proof}

\bibliographystyle{plain}
\bibliography{references}

\end{document}